\documentclass[10pt]{amsart}
\usepackage{amsmath,amssymb,latexsym,cite}
\usepackage[small]{caption}
\usepackage{graphicx,wasysym,overpic,tikz,color}
\usepackage{subfigure,color}
\usepackage{cite}
\usepackage[colorlinks=true,urlcolor=blue,
citecolor=red,linkcolor=blue,linktocpage,pdfpagelabels,
bookmarksnumbered,bookmarksopen]{hyperref}
\usepackage[italian,english]{babel}
\usepackage{units}
\usepackage{enumitem}
\usepackage[left=2.85cm,right=2.85cm,top=2.9cm,bottom=2.9cm]{geometry}
\usepackage[hyperpageref]{backref}

\numberwithin{equation}{section}
\newtheorem{theorem}{Theorem}[section]
\newtheorem{proposition}[theorem]{Proposition}
\newtheorem{lemma}[theorem]{Lemma}
\newtheorem{remark}[theorem]{Remark}

\newtheorem{definition}[theorem]{Definition}
\theoremstyle{definition}

\renewcommand{\epsilon}{\eps}

\newcommand{\R}{{\mathbb R}}

\newcommand{\eps}{\varepsilon}

\newcommand{\pnorm}[2][]{\if #1'' \left|#2\right|_p \else \left|#2\right|_{#1} \fi}

\newcommand{\RR}{\mathbb{R}}

\newcommand{\Hs}{H^s(\RR^n)}
\newcommand{\Ds}[1]{(-\Delta)^{#1}}
\newcommand{\norm}[2]{\lVert{#1}\lVert_{#2}}

\newcommand{\HL}{\Hs\cap L^{\infty}_{c}(\RR^n)}

\renewcommand{\theta}{\vartheta}

\title[Fractional logarithmic NLS]{Fractional logarithmic Schr\"odinger equations}

\author[P.\ d'Avenia]{Pietro d'Avenia}

\address{Dipartimento di Meccanica, Matematica e Management
\newline\indent 
Politecnico di Bari
\newline\indent
Via Orabona 4,  I-70125  Bari, Italy}
\email{\href{mailto:pietro.davenia@poliba.it}{pietro.davenia@poliba.it}}

\author[M.\ Squassina]{Marco Squassina}

\address{Dipartimento di Informatica \newline\indent
Universit\`a degli Studi di Verona
\newline\indent
Strada Le Grazie 15, I-37134 Verona, Italy}
\email{\href{mailto:marco.squassina@univr.it}{marco.squassina@univr.it}}

\author[M.\ Zenari]{Marianna Zenari}

\address{Dipartimento di Matematica \newline\indent
Universit\`a degli Studi di Trento
\newline\indent
Via Sommarive 14, I-38123 Povo (TN), Italy}
\email{\href{mailto:marianna.zenari@univr.it}{marianna.zenari@univr.it}}

\thanks{The first author was partially 
supported by GNAMPA project {\em Aspetti differenziali e geometrici nello studio di problemi
ellittici quasi-lineari}. The work was partially carried out during a stay of 
P.\ d'Avenia  at the University of Verona, Italy. He would like to express his gratitude to the 
Department of Computer Science for the warm hospitality.}



\begin{document}

\begin{abstract}
By means of non-smooth critical point theory
we obtain existence of infinitely many weak solutions of the 
fractional Schr\"odinger equation with logarithmic nonlinearity.\ We also investigate
the H\"older regularity of the weak solutions.
\end{abstract}

\maketitle
\section{Introduction}
\noindent
Let $s\in (0,1)$ and $n>2s$. The non-linear fractional logarithmic Schr\"odinger equation 
\begin{equation}
\label{eq:SCH}
{\rm i}\phi_t-\Ds{s}\phi+\phi\log |\phi|^2=0 \qquad\text{in $\RR\times \RR^n$}
\end{equation}
is a generalization of the classical NLS with logarithmic nonlinearity \cite{caz1}. For power type nonlinearities
the fractional Schr\"odinger equation was
derived by Laskin \cite{Lask1,Lask2,Lask3} by replacing the Brownian motion in the path integral approach
with the so called L\'evy flights. Although the equation
\begin{equation}
\label{classic-log}
{\rm i}\phi_t-\Delta\phi+\phi\log |\phi|^2=0 \qquad\text{in $\RR\times \RR^n$}
\end{equation}
has been ruled out as a fundamental quantum wave equation by 
very accurate experiments on neutron diffraction, it is currently under discussion if this equation can be adopted 
as a simplified model for some physical phenomena \cite{birula76,birula79,birula3,zlosh}.\
Its relativistic version, with D'Alembert operator in place of the Laplacian, was first proposed in \cite{rosen} by Rosen.
We refer the reader to \cite{cazNA,caz1,cazhar} 
for existence and uniqueness of solutions of the associated Cauchy problem 
in a suitable functional framework and to a study of orbital stability, with respect to radial perturbations, of the ground state solution.
%
Although the fractional Laplacian operator $(-\Delta)^s$, and more generally pseudodifferential operators, have been a classical topic of functional analysis since long ago,
the interest for such operator has constantly increased in the last few years.
Nonlocal operators such as $(-\Delta)^s$ naturally arise
in continuum mechanics, phase transition phenomena,
population dynamics and game theory, as they are the typical outcome of stochastical stabilization of L\'evy processes, see e.g.\ 
the work of Caffarelli \cite{C} and the references therein. \\
In this paper we aim to study the existence of multiple standing waves solutions to \eqref{eq:SCH}, namely
$\phi(t,x)= e^{i\omega t}u(x)$, with $\omega\in \RR$, where $u\in H^s(\R^n)$ solves the semi-linear elliptic problem
\begin{equation}\label{logeq}
\Ds{s}u+\omega u=u\log u^2 \,\,\quad\text{in $\RR^n$.}
\end{equation}
Without loss of generality we can restrict to $\omega>0$, since if $u$ is a solution of \eqref{logeq} then 
$\lambda u$ with $\lambda\neq 0$ is a solution of $\Ds{s}v+(\omega+\log \lambda^2) v=v \log v^2$.
From a variational point of view, equation \eqref{logeq} is formally 
associated with the functional $J$ on $H^s(\R^n)$ defined by
\[
J(u)={1\over 2}\int|\Ds{s/2}{u}|^2+\frac{\omega+1}{2}\int u^2-\frac{1}{2}\int u^2\log u^2.
\]
The fractional Sobolev space $H^s(\R^n)$ (see \cite{DPV}) is continuously embedded in $L^q(\R^n)$ 
for all $2\leq q\leq 2^*_s,$ where $2^*_s:=2n/(n-2s)$
and its closed subspace $H^s_{{\rm rad}}(\R^n)$ is compactly injected in $L^q(\R^n)$ for $2<q<2^*_s$ (see \cite{Pierre}).
Furthermore, by the fractional logarithmic Sobolev inequality (see \cite{CotTav}) we have
\begin{equation}
\label{logsob}
\int u^2\log\Big(\frac{u^2}{\|u\|^2_2}\Big)+\Big(n+\frac{n}{s}\log a+\log\frac{s\Gamma(\frac{n}{2})}{\Gamma(\frac{n}{2s})}\Big)\|u\|^2_2\leq 
\frac{a^2}{\pi^s}\norm{\Ds{s/2}{u}}{2}^2,\quad \text{$a>0$},
\end{equation}
for any $u\in H^s(\R^n)$. Whence, it is easy to see that $J$ satisfies this inequality
\begin{equation}
\label{MPG1}
J(u)\geq \frac{1}{2}\Big[\Big(1-\frac{a^2}{\pi^s}\Big)\norm{\Ds{s/2}{u}}{2}^2-\norm{u}{2}^2\log\norm{u}{2}^2+\Big(\omega+1+n+\frac{n}{s}\log a+\log\frac{s\Gamma(\frac{n}{2})}{\Gamma(\frac{n}{2s})}\Big)\norm{u}{2}^2\Big],
\end{equation}
for all $u\in \Hs$ and $a>0$ small. However, there are elements $u\in \Hs$ such that 
$$
\int  u^2\log u^2=-\infty.
$$
Thus, in general, the functional fails to be finite as well as of class  $C^1$. 
On the other hand, is readily seen that $J:H^s(\R^n)\to\R\cup\{+\infty\}$ is lower semi-continuous.
For this reasons we use the non-smooth critical point theory developed by Degiovanni and Zani in \cite{DegZan96,DegZan2000} for suitable classes of lower semi-continuous functionals, which is based on a generalization of the norm of the differential, the weak slope \cite{degmarz}. We say that
 $u\in\Hs$ is a weak solution to \eqref{logeq} if 
\begin{equation}
\label{def-w}
\int \Ds{s/2}{u}\Ds{s/2}{v}+\omega\int uv=\int uv\log u^2,
\,\,\,\quad\text{for all $v\in\HL$}.
\end{equation}
The main result of the paper is the following.

\begin{theorem}
\label{main}
Problem~\eqref{logeq} admits a sequence of weak solutions 
$(u_k)\subset H^s_{{\rm rad}}(\R^n)$  with $J(u_k)\to+\infty$. Furthermore, $u_k\in C^{0,2s+\sigma}(\R^n)$
for $s<1/2$ and $u_k\in C^{1,2s-1+\sigma}(\R^n)$ for $s\geq 1/2$, for some $\sigma\in (0,1).$
\end{theorem}

\noindent
The result extends to the nonlocal case the results obtained in \cite{DavMonSqua} for the existence
of multiple bound states $(u_k)\subset H^1_{{\rm rad}}(\R^n)$ for equation \eqref{classic-log}.\
Furthermore, it provides H\"older regularity of the solutions depending upon the value of $s$, following the strategy
outlined in \cite{FQT}. We point out that,  differently from \cite{FQT}, the nonlinearity $g(t)=t\log t^2$
extended to zero at $t=0$ has a very different behaviour at the origin since $g(t)/t\to-\infty$
in place of $g(t)/t\to 0$ for $t\to 0$, property which also generates, as described above, the loss of smoothness
of the functional $J$ over $H^s(\R^n)$.\ We mention that, in \cite{SSz}, a class
of non-autonomous logarithmic Schr\"odinger equations with $1$-periodic potentials was recently investigated and 
the existence of multiple solutions was obtained by splitting the energy functional into the
sum of a $C^1$ and a convex lower semi-continuous functional and using the critical point 
theory of \cite{sz}.
\vskip3pt
\noindent
The paper is organized as follows.\newline
In Section~\ref{prelim} we collect some preliminary notions and results. \newline In Section~\ref{palais} we prove that the functional
satisfies the Palais-Smale condition in the sense specified in \cite{DegZan2000}. \newline In Section~\ref{proof} we prove the
existence and the H\"older regularity of the radially symmetric weak solutions.
\vskip3pt
\noindent
Throughout the proofs the letter $C$, unless explicitly stated, will always denote a positive constant
whose value may change from line to line.
Moreover, the notation $\int$ will always denote $\int_{\R^n}.$



\section{Preliminary results}
\label{prelim}
\noindent
First, for the sake of self-containedness, we recall the definition of fractional Sobolev space and fractional Laplacian.
For any $s\in (0,1)$ the space $H^{s}(\RR^n)$ is defined as
$$
H^{s}(\RR^n):=\left\{ u\in L^2(\RR^n)\ :\ \frac{|u(x)-u(y)|}{|x-y|^{n/2+s}}\in L^2(\RR^{2n})\right\}
$$ 
and it is endowed with the norm
$$
\|u\| :=\left(\int |u|^2 +\int_{\RR^{2n}}\frac{|u(x)-u(y)|^2}{|x-y|^{n+2s}}\right)^{1/2}.
$$
Let $\mathcal{S}$ be the Schwartz space of rapidly decaying $C^\infty$ functions in $\mathbb{R}^n$. We have
\begin{definition}
For any $u\in \mathcal{S}$ and $s\in(0,1)$ the fractional Laplacian operator $\Ds{s}$ is defined as
$$
\Ds{s}u(x)=-\frac{1}{2}C(n,s)\int\frac{u(x+y)+u(x-y)-2u(x)}{|y|^{n+2s}},
$$
with
 $$
 C(n,s)=\left(\int\frac{1-\cos\zeta_1}{|\zeta|^{n+2s}} \right)^{-1}.
 $$
\end{definition}
\noindent
For functions $u$ with local H\"older continuous derivatives of exponent $\gamma>2s-1$, the integral
defining $(-\Delta)^s u$ exists finite.
\noindent
Observe that, using  \cite[Proposition 3.6]{DPV}, for every $u,v\in\Hs$ we have that
\begin{align}\label{normeq}
\int\Ds{s/2}u\Ds{s/2}v=\frac{C(n,s)}{2}\int_{\RR^{2n}}\frac{(u(x)-u(y))(v(x)-v(y))}{|x-y|^{n+2s}}.
\end{align}
We now recall some definitions and results of non-smooth critical point theory by Degiovanni and Zani \cite{DegZan2000} (see also the references therein).
\noindent Let $(X,\|\cdot\|_X)$ be a Banach space and 
$f:X\to\bar{\mathbb{R}}$ be a function.  
The (critical point) theory we follow is based on a generalized notion of the norm of the derivative, the weak slope.
First we defined it for continuous functions and then we extended it for all functions.
\begin{definition}
Let $f:X\to\RR$ be continuous and $u\in X$.
Then, $|df|(u)$ is the supremum of the $\sigma$'s in $[0,+\infty)$ such that there exist $\delta>0$ and a continuous map $\mathcal{H}:B_{\delta}(u)\times[0,\delta]\to X$, satisfying
$$d(\mathcal{H}(w,t),w)\leq t,\quad f(\mathcal{H}(w,t))\leq f(w)-\sigma t,$$
whenever $w\in B_{\delta}(u)$ and $t\in[0,\delta]$.
\end{definition}
\noindent Now, we define the function $\mathcal{G}_f: \operatorname{epi}(f) \mapsto \RR$, where $
\operatorname{epi}(f):=\left\{ (u,\lambda)
\in X\times\mathbb{R}\; \vert \; f(u)\leq\lambda\right\}$, by $\mathcal{G}_f(u,\lambda)=\lambda$. 
If on $X\times\mathbb{R}$, we consider the norm 
$\|\cdot\|_{X\times\mathbb{R}}=(\|\cdot\|_X^2 + |\cdot|^2)^{1/2}$ 
and we denote with $B_\delta(u,\lambda)$ the open ball of 
center $(u,\lambda)$ and radius $\delta>0$, we have that the function  $\mathcal{G}_f$ is continuous and Lipschitzian of 
constant $1$ and it allows to generalize the notion of weak slope for non-continuous functions $f$ as follows (see \cite[Proposition 2.3]{CamDeg}).

\begin{proposition}
For all $u\in X$ with $f(u)\in\mathbb{R}$ we have
\begin{equation*}
|df|(u)=
\begin{cases}
\displaystyle\frac{|d\mathcal{G}_f|(u,f(u))}
{\sqrt{1-|d\mathcal{G}_f|(u,f(u))^2}}
&
\hbox{if } |d\mathcal{G}_f|(u,f(u))<1,\\
+\infty &
\hbox{if } |d\mathcal{G}_f|(u,f(u))=1. 
\end{cases}
\end{equation*}
\end{proposition}
\noindent This equivalent definition allows us to study the continuous function $\mathcal{G}_f$ instead of the function $f$.
In some cases it is also useful the notion of equivariant weak slope.
\begin{definition}
Let $f$ be even with $f(0)\in\mathbb{R}$. For every $\lambda\geq f(0)$, we denote 
$|d_{\mathbb{Z}_2}\mathcal{G}_f|(0,\lambda)$ the supremum of the $\sigma$'s 
in  $[0,+\infty[$ such that there exist $\delta>0$ and a continuous map
$\mathcal{H}=(\mathcal{H}_1,\mathcal{H}_2):(B_\delta(0,\lambda)\cap\operatorname{epi}(f))\times [0,\delta] \to \operatorname{epi}(f)$,
satisfying
\begin{align*}
\| \mathcal{H}((w,\mu),t) - (w,\mu)\|_{X\times\mathbb{R}} \leq t,
\quad
\mathcal{H}_2((w,\mu),t)\leq \mu - \sigma t,
\quad
\mathcal{H}_1((-w,\mu),t)=-\mathcal{H}_1((w,\mu),t),
\end{align*}
whenever $(w,\mu)\in B_\delta(0,\lambda)\cap\operatorname{epi}(f)$ 
and $t\in [0,\delta]$.
\end{definition}
\noindent Then we can give the following
\begin{definition}
Let $c\in\mathbb{R}$. The function $f$ satisfies (epi)$_c$ condition if there exists $\varepsilon>0$ such that
\[
\inf\{|d\mathcal{G}_f|(u,\lambda)\;\vert\; f(u)<\lambda,|\lambda-c|<\varepsilon\}>0.
\]
\end{definition}
\noindent In this framework we have the following definitions.
\begin{definition}
$u\in X$ is a (lower) critical point of $f$ 
if $f(u)\in\mathbb{R}$ and $|df|(u)=0$.
\end{definition}
\begin{definition}
Let $c\in\mathbb{R}$.\ A sequence $\{u_k\}\subset X$ 
is a Palais-Smale 
sequence for $f$ at level $c$ if $f(u_k)\to c$ and
$|df|(u_k) \to 0$.
Moreover $f$ satisfies the Palais-Smale 
condition at level $c$ if every Palais-Smale 
sequence for $f$ at level $c$ admits a 
convergent subsequence in $X$. 
\end{definition}
\noindent We will apply the following abstract result (see \cite[Theorem 2.11]{DegZan2000}) that is an adaptation of the classical theorem of Ambrosetti-Rabinowitz.
\begin{theorem}
\label{astratto}
Let $X$ be a Banach space 
and $f:X\to\bar{\mathbb{R}}$ a lower semi-continuous even functional.
Assume that $f(0)=0$ and there exists a  strictly increasing sequence $\{V_k\}$ of finite-dimensional subspaces of $X$ with the following properties:
\begin{enumerate}
\item \label{it:G1} there exist a closed subspace $Z$ of $X$, $\rho >0$ and $\alpha>0$ such that $X=V_0\oplus Z$ and for every $u\in Z$ with $\|u\|_X =\rho$, $f(u)\geq\alpha$;
\item \label{it:G2} there exists a sequence $\{R_k\}\subset]\rho,+\infty[$ such that for any $u\in V_k$ with $\|u\|_X \geq R_k$, $f(u)\leq 0$;
\item  for every $c\geq\alpha$, the function $f$ satisfies the Palais-Smale condition at level $c$ and (epi)$_c$ condition;
\item \label{it:WSH} $|d_{\mathbb{Z}_2}\mathcal{G}_f|(0,\lambda)\neq 0$, whenever $\lambda\geq\alpha$.
\end{enumerate} 
Then there exists a sequence $\{u_k\}$ of critical points of $f$ such that $f(u_k)\to+\infty$.
\end{theorem}
\noindent Of course, here we need to review some theorems in \cite{DegZan2000} for the space $\Hs$.
The following result is useful to prove that our functional satisfies the hypothesis of Theorem \ref{astratto}. We know that $\HL$ is dense in $\Hs$.
Now we prove that every function in $\Hs$ can be seen as the limit of a particular sequence in $\HL$.

\begin{lemma}\label{lem5}
For every $v\in \Hs$ there exists a sequence $\{v_k\}$ in $\HL$ strongly convergent to 
$v$ in $\Hs$ with $-v^{-}\leq v_k\leq v^{+}$ a.e. in $\mathbb{R}^n$.
\end{lemma}

\begin{proof}
Assume first $v\in \HL$.
Let $\theta_k:\RR\rightarrow[0,1]$ in $C^{0,1}$ with Lipschitz constant $\lambda_k=C/k$, $\operatorname{supt}(\theta_k)\subset[-2k,2k]$, $\theta_k(s)=1$ on $[-k,k]$. Let us set $v_k:=\theta_k(v) v$. Then, observe that
$v_k(x)\to v(x)$
as $k\to\infty$
and  $-v^{-}\leq v_k\leq v^{+}$ a.e.\ in $\mathbb{R}^n$. We have
$|v_k(x)|\leq |v(x)|$ and
\begin{align*}
|v_k(x)-v_k(y)|^2
&= |(\theta_k(v(x))-\theta_k(v(y)))v(x)+(v(x)-v(y))\theta_k(v(y))|^2\\
&\leq 2(C|v(x)-v(y)|^2\norm{v}{\infty}^2+|v(x)-v(y)|^2)\leq C|v(x)-v(y)|^2.
\end{align*}
Whence $v_k\in \HL$ and,
by Lebesgue's Theorem $v_k\rightarrow v$ in $\Hs$. The general case boils down
to the previous case by arguing on $\max\{\min\{\varphi_j,v^+\},-v^-\}$ in place of $v$, where, by density, $\varphi_j\in C^\infty_c(\R^n)$  converges strongly to $v$ in $H^s(\R^n)$.
\end{proof}

\begin{remark}\rm
Arguing as in the proof of Lemma \ref{lem5}, we can get that, for every $u\in H^s_{\rm loc}(\RR^n)$,
$$
H^s_{\rm loc}(\RR^n):=\{u\in L^2_{\rm loc}(\RR^n):\|\Ds{s/2}u\|_2<+\infty\},
$$ 
and $v\in \Hs$, there exists a sequence $\{v_k\}\subset V_u$, 
$$
V_u :=\{w\in \Hs\cap L^\infty_c(\R^n):u\in L^\infty(\{x\in\R^n: w(x)\neq 0\})\},
$$ 
strongly convergent to $v$ in $\Hs$ with $-v^{-}\leq v_k\leq v^{+}$ a.e.\ (see also \cite[Theorem 2.3]{DegZan96}).
\end{remark}
\noindent 
Usually, it is not easy to compute the weak slope of a function. 
Thus, it is often useful to work with a subdifferential, for which calculus rules hold. 
\begin{definition}
For all $u\in X$ with $f(u)\in\RR$, $v\in X$ and $\epsilon>0$, we denote by $f^0_{\epsilon}$ the infimum of $r\in\RR$  such that 
there exists $\delta>0$ and a continuous map
$$
\mathcal{V}:(B_{\delta}(u,f(u))\cap \operatorname{epi}(f))\times ]0,\delta]\to B_{\epsilon}(v),
$$
such that
$$
f(w+t\mathcal{V}((w,\mu t))\leq \mu+rt,
$$
whenever $(w,\mu)\in B_{\delta}(u,f(u))\cap \operatorname{epi}(f)$ and $t\in]0,\delta].$ Then we define
$$f^0(u;v):=\sup_{\epsilon >0}f^0_{\epsilon}(u,v).$$
\end{definition} 
\noindent
As shown in \cite[Corollary 4.6]{CamDeg}, the function $f^0(u;\cdot)$ is convex, lower semicontinuous and positively homogeneous of degree 1.
We can now state the definition of the aforementioned subdifferential.
\begin{definition}
For all $u\in X$ with $f(u)\in\RR$ we define
$$\partial f(u)=\{\alpha\in X' : \langle\alpha,v\rangle\leq f^0(u;v), \ \forall v\in X\}.$$
\end{definition}
 
\noindent
Now, let us define the continuous functions
\[
g(s):=
\begin{cases}
s\log s^2 & s\neq 0\\
0 & s=0
\end{cases}
\qquad\hbox{and}\qquad
G(s):=
\begin{cases}
s^2\log s^2& s\neq 0\\
0 & s=0
\end{cases}
\]
and let 
\begin{equation}
\label{f-deff}
f(u):=\frac{1}{2}\int G(u) dx.
\end{equation} 
Note that 
$$
G(s)=2\int_0^s (g(t)+t) dt.
$$
We have the following preliminary result. 

\begin{proposition}\label{convergzz}
	If $u\in H^s_{\rm loc}(\RR^n)$ we have that:
	\begin{enumerate}
		\item \label{en1}for every $v\in \HL$, $g(u)v\in L^1(\RR^n)$;
		\item \label{en2}let $v \in \Hs$ and assume that $(g(u)v)^{+}\in L^1(\RR^n)$ or $(g(u)v)^{-}\in L^1(\RR^n)$, then there exists a sequence $\{v_k\}$ in $\HL$ strongly convergent to $v$ in $\Hs$ with
		$$\lim_{k\rightarrow\infty}\int g(u)v_k=\int g(u)v.$$
	\end{enumerate} 
\end{proposition}

\begin{proof}
	If $v\in \HL$, for $\delta\in (0,\frac{N+2s}{N-2s})$, we have
	\begin{align*}
		\int|g(u)v| &\leq\norm{v}{\infty}\Big(\int_{{\rm spt}(v)\cap{|u|\leq1}}|g(u)|+\int_{{\rm spt}(v)\cap{|u|>1}}|g(u)|\Big) \\
		&\leq C\Big(1+\int_{{\rm spt}(v)\cap{|u|>1}}|u|^{1+\delta}\Big)<+\infty,
	\end{align*}
	then we have (\ref{en1}).
	To prove (\ref{en2}) we argue as in \cite[Theorem 2.7]{DegZan96}. Let us assume for instance that $(g(u)v)^{+}\in L^1(\RR^n)$ (if $(g(u)v)^{-}\in L^1(\RR^n)$ the proof is similar). By Lemma \ref{lem5}, there is a sequence $\{v_k\}$ in $\HL$ such that $v_k\rightarrow v$ in $\Hs$ and $-v^{-}\leq v_k\leq v^{+}$ a.e. in $\RR^n$ and, by (\ref{en1}), for every $k$, $g(u)v_k\in L^1(\RR^n)$. But
	\[
	g(u)v_k=g(u)^+ v_k - g(u)^- v_k\leq g(u)^+ v^+ + g(u)^- v^- = (g(u)v)^+\in L^1(\RR^n)
	\]
	and by Fatou's Lemma we have
	$$\limsup_{k}\int g(u)v_k\leq \int g(u)v.$$
	Hence, if $\int g(u)v=-\infty$ we conclude, otherwise we have that $g(u)v\in L^1(\RR^n)$ since 
	\[
	\int |g(u)v|
	=\int (g(u)v)^+ + \int (g(u)v)^-
	=2\int (g(u)v)^+ - \int g(u)v,
	\]
	and $|g(u)v_k|\leq|g(u)v|$. Thus, by Lebesgue's Theorem we conclude.
\end{proof}
\noindent
Moreover we have the following theorem, whose proof is the same of \cite[Theorem 3.1]{DegZan2000}.
	\begin{theorem}\label{DZTh.3.1}
		Let $u\in \Hs$ with $f(u)\in \RR$.
If $\partial (-f)(u)\neq\emptyset$, then 
$$
\sup\Big\{\int (-g(u)-u)v:v\in \HL, \|v\|\leq 1 \Big\}<+\infty,
$$
and hence $-g(u)-u\in H^{-s}(\RR^n)$ upon identitication of $-g(u)-u$ with its unique extension. Furthermore 
$\partial (-f)(u)=\{-g(u)-u\}$ and 
			for all  $v\in \Hs$ with  $(g(u)v)^+\in L^1(\RR^n)$
			or $(g(u)v)^{-}\in L^1(\RR^n)$, it holds
			$$
			\langle -g(u)-u,v \rangle=\int (-g(u)-u)v.
			$$
			In particular, this holds true for every $v\in \HL$.
	\end{theorem}
\noindent
Finally, in our case the (epi)$_{c}$ condition and (\ref{it:WSH}) of 
Theorem \ref{astratto} is easy to prove thanks to the following theorem. 
	\begin{theorem}
		\label{DZTh.3.4}
		Let $(u,\lambda)\in \operatorname{epi}(f)$ with $\lambda>f(u)$. Then $|d \mathcal{G}_f|(u,\lambda)=1$ and, furthermore, $|d_{\mathbb{Z}_2}\mathcal{G}_f|(0,\lambda)=1$ for all $\lambda>f(0).$
	\end{theorem}
\noindent
The proof can be obtained arguing as in \cite[Theorem 3.4]{DegZan2000}.

\section{Palais-Smale condition}
\label{palais}
\noindent
In this section we prove that $J$ 
satisfies the Palais-Smale condition, thus we can apply Theorem \ref{astratto} to prove the existence of infinitely many weak solutions to \eqref{logeq}, namely functions $u\in\Hs$ such that 
\eqref{def-w} holds for any $v\in\HL$.
Notice that,  that if $u\in\Hs$ and $v\in\HL$, by Proposition \ref{convergzz} we can consider
\begin{equation} 
\label{equat-var}
\langle J'(u),v\rangle=\int \Ds{s/2}{u}\Ds{s/2}{v}+\omega\int uv-\int uv\log u^2.
\end{equation}

\noindent
We will need the following

\begin{proposition}\label{linkder-pend}
Let $u\in \Hs$ with $J(u)\in \RR$ and $|dJ|(u)<+\infty$. Then the following facts hold:
\begin{enumerate}
\item \label{lp1}$g(u)\in L^1_{\rm loc}(\RR^n)\cap H^{-s}(\RR^n)$ and 
$|\langle \alpha_u,v\rangle|\leq|dJ|(u)\norm{v}{}$ for all $v\in \Hs$, where
\begin{equation}
\label{alfa-def}
\langle \alpha_u,v\rangle:=\int \Ds{s/2}{u}\Ds{s/2}{v}+(\omega + 1)\int uv+\langle -g(u)-u,v \rangle.
\end{equation}
In particular, for every $v\in \HL$ we have
$$
|\langle J'(u),v\rangle|\leq|dJ|(u)\norm{v}{}.
$$
\item \label{tests} if $v\in\Hs$ is such that $(g(u)v)^+\in L^1(\RR^n)$ or $(g(u)v)^-\in L^1(\RR^n)$, then $g(u)v\in L^1(\RR^n)$ and 
identity \eqref{equat-var} holds.
\end{enumerate}
\end{proposition}

\begin{proof}
As in the proof of (\ref{en1}) in Proposition \ref{convergzz} we have $g \in L^1_{{\rm loc}}(\R^n).$  
Moreover we can write our functional as $J(u)=S(u)-f(u)$,
where $f$ is as in \eqref{f-deff} and 
$$
S(u)=\frac{1}{2}\int |\Ds{s/2}u|^2+\frac{\omega+1}{2}\int u^2.
$$
Using the  properties of the weak slope (see e.g.\ \cite[Theorem 4.13]{CamDeg}), we can see that $\partial J(u)\neq\emptyset$ and, by the calculus rule, $\partial J(u)\subseteq \partial S(u)+\partial (-f)(u)$ (see \cite[Corollary 5.3]{CamDeg}), the $\partial (-f)(u)$
is nonempty too. By Theorem \ref{DZTh.3.1}
we obtain that $\partial (-f)(u)= \{-u-g(u)\}$.  Since $S$ is $C^1$, again by \cite[Corollary 5.3]{CamDeg}, 
$\partial S(u)=\{S'(u)\}$ and then by 
\cite[Theorem 4.13, (iii)]{CamDeg}, we have $\partial J(u)=\{\alpha_u\}$ and
\[
|dJ|(u)\norm{v}{}\geq\min\{\|\beta\|_{H^{-s}}:\beta\in\partial J(u)\}\norm{v}{}=\|\alpha_u\|_{H^{-s}}\norm{v}{}\geq |\langle \alpha_u,v\rangle|.
\]
The second part follows by using \eqref{alfa-def} and
assertion \eqref{en2} of Proposition \ref{convergzz}.
\end{proof}


\begin{remark}\rm
\label{semicont}
It is readily seen that 
$J$ is lower semi-continuous, see e.g.\ \cite[Proposition 2.2]{DavMonSqua} for the details.
Alternatively, one can observe that there exist $q>2$ and $C>0$ such that $G(s)\leq C|s|^{q}$ for all $s\in\R$.
Then, the assertion follows by a variant of Fatou's lemma.
\end{remark}

\noindent
\noindent
Finally, we can prove the following

\begin{proposition}
\label{comp-cps}
$J|_{H^s_{\rm rad}(\RR^n)}$ satisfies the Palais-Smale condition at level $c$ for every $c\in\mathbb{R}$.
\end{proposition}
\begin{proof}
Let $\{u_k\}\subset \Hs$ be a Palais-Smale sequence of $J$, i.e.
$J(u_k)\to c$ and $|dJ|(u_k) \to 0$, thus by Proposition~\ref{linkder-pend} we have that $\langle J'(u_k),v\rangle=o(1)\|v\|$
for any $v\in \HL$.
It is easy to see that if $u\in\Hs$, then $(u^2\log u^2)^+\in L^1(\RR^n)$, thus by  \eqref{tests}
Proposition~\ref{linkder-pend}, the $u_k$ are  admissible test functions
in equation~\eqref{equat-var} and
\begin{equation}
\label{eq:boundL2}
\| u_k\|_{2}^2 = 2 J(u_k) - \langle J'(u_k),u_k\rangle
\leq 2c+o(1) \|u_k\|.
\end{equation}
Using \ref{logsob}, we have that
\begin{align*}
\|u_k\|^2 &= 2J(u_k)-\omega \| u_k\|_{2}^2+\int u_k^2\log u_k^2\\
&\leq 2c + \frac{a^2}{\pi^s}\norm{\Ds{s/2}u_k}{2}^2+\norm{u_k}{2}^2\log\norm{u_k}{2}^2
-\Big(\omega+n+\frac{n}{s}\log a+\log \frac{s\Gamma(\frac{n}{2})}{\Gamma(\frac{n}{2s})}\Big)\norm{u_k}{2}^2.
\end{align*}
Thus for $a>0$ and $\delta>0$ small and by \eqref{eq:boundL2} we have
\[
\|u_k\|^2
\leq 
C+o(1) \|u_k\|^{1+\delta} + o(1)\|u_k\|
\]
and so $\{u_k\}$ is bounded in $\Hs$.
\noindent
Let $\{u_k\}$ now be a Palais-Smale sequence for $J$ in $H^s_{\rm rad}(\RR^n)$. 
By the boundedness of $\{u_k\}$ and thanks to the compact embedding $H^s_{\rm rad}(\RR^n)\hookrightarrow L^p(\RR^n)$ for $2<p<2_s^*$, we have that up to a subsequence, there exists $u\in H^s_{\rm rad}(\RR^n)$ such that
\[
u_k \rightharpoonup u  \hbox{ in } H^s_{\rm rad}(\RR^n),   
\qquad
u_k \to u  \hbox{ in } L^p(\RR^n),\,\,\,\,\, 2<p<2^*_s,
\qquad
u_k \to u  \hbox{ a.e. in } \RR^n.  
\]
We want to prove that for all $v\in\HL$
\begin{equation}
\label{solprobl}
\int\Ds{s/2}{u}\Ds{s/2}{v}
+\omega\int u v
=\int u v \log u^2.
\end{equation}
So, fixed $v\in\HL$, let us consider 
$\theta_R(u_k)v$,
where, given $R>0$, $\theta_R:\mathbb{R}\to[0,1]$ is a $C^{0,1}$ function such that $\theta_R(s)=1$ for $|s|\leq R$,  $\theta_R(s)=0$ for $|s|\geq 2R$ and $|\theta'_R(s)|\leq C/R$ in $\mathbb{R}$.
Obviously, as in Lemma \ref{lem5} we have that $\theta_R(u_k)v\in\HL$.
Thus, by~\eqref{equat-var} and \eqref{normeq} we have
\begin{align*}
\langle J'(u_k),\theta_R(u_k)v\rangle
&=
\int\Ds{s/2}{u_k}\Ds{s/2}{(\theta_R(u_k)v)}
+\omega\int\theta_R(u_k)u_k v 
-\int \theta_R(u_k)u_k v\log u_k^2\\
&=
\frac{C(n,s)}{2}
\int_{\RR^{2n}}\frac{(u_k(x)-u_k(y))(\vartheta_R(u_k(x))v(x)-\vartheta_R(u_k(y))v(y))}{|x-y|^{n+2s}}
\\
&\qquad
+\omega\int\theta_R(u_k)u_k v 
-\int \theta_R(u_k)u_k v\log u_k^2\\
&=
\frac{C(n,s)}{2} \int_{\RR^{2n}}\frac{\theta_R(u_k(x))(u_k(x)-u_k(y))}{|x-y|^{\frac{n+2s}{2}}}\frac{(v(x)-v(y))}{|x-y|^{\frac{n+2s}{2}}}\\
&\qquad
+\frac{C(n,s)}{2} \int_{\RR^{2n}}\frac{v(y)(\theta_R(u_k(x))-\theta_R(u_k(y)))(u_k(x)-u_k(y))}{|x-y|^{n+2s}}\\
&\qquad
+\omega\int\theta_R(u_k)u_k v 
-\int \theta_R(u_k)u_k v\log u_k^2.
\end{align*}
Then, we obtain
\begin{align*}
&\Big| \frac{C(n,s)}{2}\int_{\RR^{2n}}\frac{\theta_R(u_k(x))(u_k(x)-u_k(y))}{|x-y|^{\frac{n+2s}{2}}}\frac{(v(x)-v(y))}{|x-y|^{\frac{n+2s}{2}}} 
+\omega\int_{\RR^n}\theta_R(u_k)u_k v 
-\int_{\RR^n} \theta_R(u_k)u_k v\log u_k^2  \\
&\quad
-\langle J'(u_k),\theta_R(u_k)v\rangle\Big|
\leq\|v\|_{\infty}\frac{C}{R}\int_{\RR^{2n}}\frac{|u_k(x)-u_k(y)|^2}{|x-y|^{n+2s}}
\leq \frac{C}{R}.
\end{align*}
Since
$$
\theta_R(u_k(x))\frac{(u_k(x)-u_k(y))}{|x-y|^{\frac{n+2s}{2}}}\text{ is bounded in } L^2(\RR^{2n}), \quad \frac{(v(x)-v(y))}{|x-y|^{\frac{n+2s}{2}}}\in L^2(\RR^{2n}),
$$
and
\[
\theta_R(u_k(x))\frac{(u_k(x)-u_k(y))}{|x-y|^{\frac{n+2s}{2}}}
\to
\theta_R(u(x))\frac{(u(x)-u(y))}{|x-y|^{\frac{n+2s}{2}}}
\quad
\text{a.e. } (x,y)\in \RR^{2n}
\text{ as } k\to+\infty
\]
then
\[
\int_{\RR^{2n}}\frac{(v(x)-v(y))}{|x-y|^{\frac{n+2s}{2}}}
\theta_R(u_k(x))\frac{(u_k(x)-u_k(y))}{|x-y|^{\frac{n+2s}{2}}}
\to
\int_{\RR^{2n}}\frac{\theta_R(u(x))(u(x)-u(y))}{|x-y|^{\frac{n+2s}{2}}}\frac{(v(x)-v(y))}{|x-y|^{\frac{n+2s}{2}}}
\] 
as $ k\to+\infty$. In the same way, taking into account
that $\{\theta_R(u_k) u_k \log u_k^2 \}$ 
is bounded in $L^2_{\rm loc}(\R^n)$ and since $ \theta_R(u_k) u_k \log u_k^2 \to \theta_R(u) u \log u^2$
a.e. in $\RR^n$ and , 
we have
\[
\Big| \frac{C(n,s)}{2} \int_{\RR^{2n}}\frac{\theta_R(u(x))(u(x)-u(y))}{|x-y|^{\frac{n+2s}{2}}}\frac{(v(x)-v(y))}{|x-y|^{\frac{n+2s}{2}}} 
+\omega \int \theta_R(u) u v
-\int \theta_R(u) u v \log u^2 \Big|
\leq \frac{C}{R}.
\]
Thus, letting $R\to\infty$, \eqref{normeq}  yields \eqref{solprobl}.
Moreover, see again Remark~\ref{semicont}, we have that
$$
\limsup_k \int u_k^2\log u_k^2\leq \int u^2\log u^2.
$$
Hence, since $\langle J'(u_k),u_k \rangle \to 0$ and choosing  $v=u$ 
in~\eqref{solprobl}, we get
\begin{align*}
\limsup_k (\|\Ds{\frac{s}{2}}{u_k}\|_2^2 +\omega \|u_k\|_2^2)
=\limsup_k \int u_k^2 \log u_k^2\leq \int u^2\log u^2= \|\Ds{\frac{s}{2}}{u}\|_2^2 +\omega \|u\|_2^2,
\end{align*}
which implies the convergence of $u_k\to u$ in $H^s_{\rm rad}(\RR^n)$.
\end{proof}

\section{Proof of Theorem \ref{main}}
\label{proof}

\subsection{Proof for existence}

\noindent
To prove the existence of sequence $\{u_k\}\subset \Hs$
of weak solutions to \eqref{logeq} with $J(u_k)\to+\infty$,
we will apply Theorem~\ref{astratto} with $X=H^s_{\rm rad}(\RR^n)$.
By Proposition~\ref{comp-cps} we know that $J$ satisfies 
the Palais-Smale condition. Furthermore, 
%
%
by Theorem \ref{DZTh.3.4} we have that 
$J$ satisfies (epi)$_c$ and (\ref{it:WSH}) of Theorem \ref{astratto}. Hence, we only have to prove that $J$ satisfies also the 
geometrical assumptions. Obviously, 
$J(0)=0$, and by \eqref{MPG1},
$J(u)\geq c \| u\|^2$,
for a suitable $a$ and  if $\|u\|_{2}$ are sufficiently small.
Then, if we take $Z=H^s_{\rm rad}(\RR^n)$ and $V_0=\{0\}$ we have (\ref{it:G1}).
Finally, let $\{V_k\}$ a strictly increasing 
sequence  of finite-dimensional 
subspaces of $H^s_{\rm rad}(\RR^n)$.
Since any norm is equivalent on any $V_k$, if $\{u_m\}\subset V_k$ is such 
that $\|u_m\|\to +\infty$, then also $\mu_m:=\|u_m\|_{2} \to +\infty$. Set
now $u_m=\mu_m w_m$, where $w_m=\|u_m\|_{2}^{-1}u_m$. 
Thus $\|w_m\|_{2}=1$, $\|\Ds{s/2}{w_m}\|_{2}\leq C$
and $\|w_m\|_{\infty}\leq C$, and so
\begin{align*}
J(u_m)
=\frac{\mu_m^2}{2}  \Big(  \|\Ds{s/2}{w_m}\|_{2}^2 
+ \omega + 1 - \log \mu_m^2
- \int w_m^2 \log w_m^2 \Big)
\leq  \frac{\mu_m^2}{2} (C - \log \mu_m^2)\to-\infty.
\end{align*}
Thus, there exist $\{R_k\}\subset]\rho,+\infty[$ 
such that for $u\in V_k$ with $\|u\| \geq R_k$, $J(u)\leq 0$ and the condition \eqref{it:G2} is satisfied.

\subsection{Proof for regularity}
To prove the regularity we follow \cite{FQT}. First of all we define
\[
\mathcal{W}^{\beta,p}(\mathbb{R}^n)
=\{u\in L^p(\mathbb{R}^n):\mathcal{F}^{-1}[(1+|\xi|^\beta)\hat{u}]\in L^p(\mathbb{R}^n)\}.
\]
For the properties of this space, we refer to \cite{FQT}.
Now, let $u\in \Hs$ be a solution of \eqref{logeq} and $\{r_i\}$ a strictly decreasing sequence of positive constants with $r_0=1$. Let $B_i=B(0,r_i)$ and define
\[
h(x)= u(x) \log u^2(x).
\]
We have that
\begin{equation}
\label{accaelle}
|h|\leq C_\delta (|u|^{1-\delta}+|u|^{1+\delta})
\end{equation}
for all $\delta\in(0,1)$.
Now let $\eta_1\in C^\infty (\mathbb{R}^n)$, $0\leq\eta_1 \leq 1$, $\eta_1=0$ in $B_0^c$, $\eta_1=1$ in $B_{1/2}:=B(0,r_{1/2})$ with $r_1<r_{1/2}<r_0$ and $u_1$ be the solution of
\[
\Ds{s}u_1 + \omega u_1 = \eta_1 h
\quad
\hbox{in }\mathbb{R}^n 
\]
namely, $u_1=\mathcal{K}*(\eta_1 h)$, where $\mathcal{K}(x)=\mathcal{F}^{-1}(1/(\omega+|\xi|^{2s}))$ is the Bessel kernel.
Then
\[
\Ds{s}(u-u_1) +  \omega (u-u_1) = (1-\eta_1) h
\quad
\hbox{in }\mathbb{R}^n
\]
and so
\[
u-u_1=\mathcal{K}*[(1-\eta_1) h].
\]
By Sobolev embedding Theorem, $u\in L^{q_0}(\mathbb{R}^n)$ with $q_0=2n/(n-2s)$. Moreover, by \eqref{accaelle}, \cite[Theorem 3.3]{FQT} and 
H\"older inequality we have that for a.e. $x\in B_1$
	\begin{equation}
	\label{38}
	|u(x) - u_1(x)| 
	\leq C 
	(\|\mathcal{K}\|_{L^{s_0}(B_{r_{1/2} - r_1}^c)}
	\|(1-\eta_1)^{1/(1-\delta)}u\|_{q_0}^{1-\delta}
	+
	\|\mathcal{K}\|_{L^{s_1}(B_{r_{1/2} - r_1}^c)}
	\|(1-\eta_1)^{1/(1+\delta)}u\|_{q_0}^{1+\delta})
	\end{equation}
	where $s_0={q_0}/{(q_0-1+\delta)}$, $s_1={q_0}/{(q_0-1-\delta)}$ and 
	$\delta<\min\{1,(n+2s)/(n-2s)\}$.
	In fact,
	\begin{align*}
	|u(x) - u_1(x)| 
	& \leq
	\int_{B_{1/2}^c} |\mathcal{K}(x-y)| |(1-\eta_1(y))h(y)|dy  \\
	& \leq  
	C\Big(\int_{B_{1/2}^c(x)} |\mathcal{K}|^{s_0}\Big)^{1/{s_0}}
	\|(1-\eta_1)^{1/(1-\delta)}u\|_{q_0}^{1-\delta}   \\
	& +
	C\Big(\int_{B_{1/2}^c(x)} |\mathcal{K}|^{s_1}\Big)^{1/{s_1}}
	\|(1-\eta_1)^{1/(1+\delta)}u\|_{q_0}^{1+\delta}
	\end{align*}
and $B_{1/2}^c(x)\subset B_{r_{1/2} - r_1}^c$.	Notice that, the same argument shows that
for all $z\in\R^n$ and for a.e. $x\in B_1(z)$
	\begin{equation*}
	|u(x) - u_1(x)| 
	\leq C 
	(\|\mathcal{K}\|_{L^{s_0}(B_{r_{1/2} - r_1}^c)}
	\|u\|_{q_0}^{1-\delta}
	+
	\|\mathcal{K}\|_{L^{s_1}(B_{r_{1/2} - r_1}^c)}
	\|u\|_{q_0}^{1+\delta}).
	\end{equation*}
and thus, in turn, since the right hand side is independent of the point $z$, it follows that $u-u_1\in L^\infty(\R^n)$.
Since $u\in L^{q_0}(\mathbb{R}^n)$ and $B_0$ is bounded, we have that $\eta_1 h\in L^{p_1}(\mathbb{R}^n)$ with $p_1=q_0/(1+\delta)$.
Then $u_1\in\mathcal{W}^{2s,p_1}(\mathbb{R}^n)$. 

If $n<6s$, then, in \eqref{accaelle}, we take 
\[
\delta<\min\left\{1,\frac{6s-n}{n-2s}\right\}
\]
and so $p_1>n/(2s)$.

If $n\geq 6s$, we have that $p_1<n/(2s)$ and we proceed as follows. By Sobolev embedding and \eqref{38} we have that $u\in L^{q_1}(B_1)$ with $q_1=p_1n/(n-2sp_1)$. 
Then we repeat the procedure, namely we consider $\eta_2\in C^\infty (\mathbb{R}^n)$, $0\leq\eta_2 \leq 1$, $\eta_2=0$ in $B_1^c$, $\eta_1=1$ in $B_{3/2}:=B(0,r_{3/2})$ with $r_2<r_{3/2}<r_1$, getting that $u_2=\mathcal{K}*(\eta_2 h) \in \mathcal{W}^{2s,p_2}(\mathbb{R}^n)$ with $p_2=q_1/(1+\delta)$.

If $n<10 s$ and in \eqref{accaelle} we take
\[
\delta< \frac{-(n-4s)+\sqrt{4s(n-s)}}{n-2s}
\]
and we have that $p_2>n/(2s)$.

If $n\geq10s$, then $p_2<n/(2s)$ and we iterate this procedure. Straightforward calculations show that
\begin{equation}
\label{qj+1}
\frac{1}{q_{j+1}}=
\frac{1}{q_1} + \left(\frac{1}{q_1} - \frac{1}{q_0}\right)\sum_{i=1}^{j}(1+\delta)^i
=\frac{(1+\delta)^{j+1}}{q_0} - \frac{2s}{n}\sum_{i=0}^{j}(1+\delta)^i
\end{equation}
and, using \eqref{qj+1}, that $p_j>n/(2s)$ is equivalent to
\begin{equation}
\label{pazza}
(n-2s)(1+\delta)^j -4s\sum_{i=1}^{j-1} (1+\delta)^i -4s<0.
\end{equation}
From \eqref{pazza} we get that, if 
\begin{equation}
\label{condsteps}
2(2j-1)s\leq n< 2(2j+1)s,
\end{equation}
then we can take $\delta$ small enough such that $p_j>n/(2s)$. Of course, this procedure stops in $j$ steps with $j$ that satisfies \eqref{condsteps}.

%
Thus, if $\ell$ is such that $p_{\ell}> n/(2s)$, since $u_{\ell}\in \mathcal{W}^{2s,p_{\ell}}(\mathbb{R}^n)$, by Sobolev imbeddings (see \cite[Theorem 3.2]{FQT}), we have that $u_{\ell}\in C^{0,\mu}(\R^n)$ for  $\mu>0$ small enough. 
Moreover, we can estimate $|u-u_\ell|$ in $B_{\ell}$ as in \eqref{38} and, using the smoothness of $\mathcal{K}$ away from the origin (see \cite[Theorem 3.3]{FQT}) and since $|x-y|\geq C >0$ for $x\in B_{\ell}$ and $y\in B_{\ell-1/2}^c$ we have that for $x\in B_{\ell}$
\[
|\nabla (u-u_{\ell}) (x)|
\leq
\int_{B_{\ell-1/2}^c} |\nabla\mathcal{K}(x-y)| 
(|u(y)|^{1-\delta}+|u(y)|^{1+\delta})
\leq
C(n,s,\|u\|).
\]
Then $u-u_{\ell}\in W^{1,\infty}(B_{\ell})$ and so, $u-u_{\ell}\in C^{0,\mu}(B_{\ell})$. Then $u\in C^{0,\mu}(B_{\ell})$ and the $C^{0,\mu}$-norm depends on $n,s,\|u\|_{H^s}$ and on the finite sequence $r_0,\ldots,r_{\ell}$. 
Moving $B_{\ell}$ around $\mathbb{R}^n$ we can recover it, obtaining that $u\in C^{0,\mu}(\mathbb{R}^n)$ and since in addition $u\in L^{q_0}(\mathbb{R}^n)$ we get that $u(x)\to 0$ as $|x|\to + \infty$. \newline
We now claim that for all $a,b\in\R$ with $a<b$ and any $\delta\in (0,1)$ we have $g\in C^{0,1-\delta}([a,b])$.
Indeed, if  $s,t\in [a,b]$ with e.g. $t>s>0$, then we have 
\[
|g(t)-g(s)| \leq \int_s^t |g'(\xi)|d\xi\leq 2\int_s^t (|\log\xi|+1)d\xi\leq C\int_s^t \xi^{-\delta} =C(t^{1-\delta}-s^{1-\delta})\leq C(t-s)^{1-\delta}.
\]
By symmetry the same inequality
holds for negative $s,t\in [a,b]$. If $s,t\in [a,b]$ with e.g. $s\leq 0\leq t$, we get
\begin{equation*}
|g(t)-g(s)|
\leq |g(t)| + |g(s)|
\leq Ct^{1-\delta}+C(-s)^{1-\delta}\leq C(t-s)^{1-\delta},
\end{equation*} 
proving the claim. Then the regularity assertions of Theorem~\ref{main} follow by 
arguing as in \cite[p.\ 1251]{FQT}.

\bigskip
\medskip

\end{document}